\documentclass[12pt]{amsart}
\usepackage{amssymb}
\usepackage{amsfonts}
\usepackage{amssymb,latexsym}
\usepackage{enumerate}
\usepackage{mathrsfs}

\usepackage{graphicx}
\usepackage[all]{xy}

\newtheorem{thm}{Theorem}[section]

\newtheorem{cor}[thm]{Corollary}

\theoremstyle{definition}

\theoremstyle{remark}

\begin{document}

\title[{\normalsize {\large {\normalsize {\Large {\LARGE }}}}}  Fermat's  Last  theorem  and  its  another proof
]{Fermat's  Last  theorem and  its another  proof }

\author{Kim YangGon, Kim SooGon,Kim SeungKon, Kim ChangKon }

\address{emeritus professor,
	 Department of Mathematics,  
	Jeonbuk National University, 567 Baekje-daero, Deokjin-gu, Jeonju-si,
	Jeollabuk-do, 54896, Republic of Korea.}

\

\

\email{ kyk1.chonbuk@hanmail.net, kyk2jeonbuk@naver.com}

\address{former president of  Jeonbuk National  University.}
\
\email{   -     }

\address{emeritus professor, Department of Physics, Jeonbuk National University.}
\

\email{skkim@jbnu.ac.kr}

\address{doctor of  the clinic 'Kim AnGwa', Iksan city, Jeollabuk-do.}
\

\email{Kck5457@hanmail.net}


\begin{abstract}
We announce here that Fermat's Last theorem was solved, but there is an easy proof  of it on the basis of  elemetary undergraduate mathematics.
We shall disclose such an easy proof and discuss more about it.

\end{abstract}

\maketitle

\section{introduction }

\

\

\large{

It  is  a well known fact that  Fermat's  Last theorem was  solved  by  Andrew Wiles(1953.4.11 - ), but  his paper is lengthy and hard. So  probably many mathematicians  still want  
 an easy one even for  undergraduates  or  civilians to understand without difficulty.\newline

This paper could be  a sort  of  an  answer to their request. Hence  it is  hoped  that  this paper could be really  helpful  for  them. \newline

Of course there could exist   papers  easier than this paper,\newline

for example the readers might surf on the arxiv. org to  view  many uptodate papers including even 1-page paper proving  the Fermat Last Theorem.\newline

 But we found out  therein some serious flaws or gaps  before reaching the final goal.\newline

 In this paper we shall proceed in the following  order.\newline

First, we  are going to explain what  the  Fermat's Last theorem is. \newline

Secondly, we shall give an easy solution even for  undergraduates  or  ordinary people to comprehend.\newline

Finally we finish this paper with a  concluding remark dealing with some  Diophantine problem  which is  equivalent  to  the Fermat's Last theorem .\newline

\

\section{What is the Fermat's  Last theorem?}

\
\

Let  us consider  the quadratic  equation $X^2+ Y^2= Z^2$ defined in the  set  of natural  numbers. It has a  solution $X= 3, Y= 4, Z= 5 $ as an example. As is well known,this equation  is obtained  from the Pythagorus theorem.\newline

Here we ask a question. What about  the equation\newline

  (2.1)  $ X^n+ Y^n= Z^n $ \newline

for any  integer value $ n >2$  defined in the ring of integers $\mathbb{Z}$? \newline

Does it have a  solution  at all with $X \neq 0, Y \neq 0, Z  \neq 0 ?$ The answer is 'No'  according to  the French mathematician
Pierre de Fermat(1601.8.17- 1665.1.12).\newline

 But  his proof  was not given  except that  he  said\newline

 " I have discovered a truly remarkable proof of this theorem which this margin is too small to contain." \newline

 In 1995 the first proof was given by Andrew Wiles. His proof uses the Taniyama-Shimura  conjecture  relating to elliptic curves  and modular forms of them, and his paper is  thought  to be hard and long.\newline

We intend to  give a short and easy proof  for this problem by making use of reduction to absurdity.\newline

It suffices to consider  only  positive integer solutions for  this equation. \newline

 We may  also  assume  that $n$ is any fixed odd prime number because  the  exponent  has  a unique  factorization of prime numbers and $X^4+ Y^4= Z^4 $ was solved by Ferma himself. Furthermore we may assume that  $X,Y$, and  $Z$ are relatively prime . \newline

If $n$ is a  prime number , then  one of  $X$ and $Y$  cannot  be  a multiple of  $n$ by our assumption.          \newline

Suppose that we have  a solution  for the  equation (2.1). If we divide both sides by $X^n$ which may be assumed to be  no multiple of $n$  without  loss of generality,we get  the equation of the form  $1+ (Y/X)^n= (Z/X)^n$.\newline

 Now if we put 

$A=Y/X, 1+ B= Z/X$, where $A$  and $B$ are rational numbers in  the field $\mathbb{Q}$, then  there  arises an equation of the form \newline

(2.2) $1+  A^n= (1+ B)^n$,\newline

where $A$ and $B$  must  be  positive  rational  numbers. Of course $A$ must be exactly greater than $B$ for the equality of (2.2) to hold,i.e., $A$$>$$B > 0$. \newline

The converse also holds, i.e.,if there is a solution  of  (2.2)  in $\mathbb {Q}$
, then there must be a solution of integers for the equation (2.1).\newline

There are infinitely many solutions for $n= 1,2$. \newline

For $n= 1$, there  are  obviously  many  solutions.\newline

 For $n= 2$, it says nothing but the Pythagorean theorem ,and so we  consider  the equation  of the  form  $1+ A^n= (A+ \alpha)^n$. We then have $1+ A^2= A^2+ 2A\alpha+ \alpha^2$, so that we  may get solutions $A= (1-\alpha^2)/ 2\alpha $  for  any  rational  number  $\alpha$.\newline

For $n \geq 3$, however, there is no solution at all. We shall prove this  fact  in the  next  section 3 .

\

\section{Proof of the Fermat's last theorem}
\
\

Now we assume henceforth once and for all  that  there is a solution in $\mathbb{Q}$ for the equation\newline

 (3.1)  $ 1+ A^n= (1+ B)^n= B^n+ \binom{n}{1}B^{n-1}+ \cdots+ \binom{n}{n-1} B + 1$, \newline

where $A$ and  $B$  are  positive rational numbers  as manifested just above. \newline

\begin{thm}

We have no solution  of  the designated equation (2.1).

\end{thm}

\begin{proof}
We  note  first thing that the solution of (2.1), if  any,   may be obtained  roughly  from the congruence equation  \newline

(3.2) $1+ A^n \equiv (1+ A)^n \equiv 1+ A \equiv (1+ B)^n \equiv 1+ B^n \equiv 1+ B$ modulo $(n)$ \newline

as follows.\newline

(3.3)  $A=$$ t \over s$$,B= (bt+ ln)/(bs+ kn)$ for some integers $b,l,k,s,$ and $t$  because $A^n\equiv A$  and  $B^n\equiv B$  in the multiplicative group  $\mathbb{Z}_n^{\times}$ and so in $\mathbb{Z}_n$.\newline

In particular  we may assume without loss of generality that  $b,s,t$  are positive integers.We consider $b=1$ at first.\newline

Here  the fractions of  $A$  and  $B$  are assumed to be irreducible ones and  we would like to know  the relationship of  these  in  (3.3)  with  $n$  in more detail.\newline

As  a matter of fact we are dealing  here with  the localization  $S^{-1}\mathbb{Z}$  of  the ring of integers $\mathbb {Z}$ at  a  prime ideal  $(n)$, where  $S$ denotes  a multiplicative set  $\mathbb {Z}- (n)$. \newline

At the same time  we may  consider by the universal mapping property  the composite map: $\mathbb{Z} \rightarrow S^{-1}\mathbb{Z} \rightarrow \mathbb{Z}_n$ of ring homomorphisms,where  $\mathbb{Z}_n$ denotes  the finite field  consisting of  $n$ numbers $\{0,1,2,\cdots,(n-1)\}$.\newline

From (2.2), we see that $s+ nk$  must  be  a  multiple  of  $s$ and  the numerator $ t $ of $A$  may happen to be a multiple of the numerator  $t+ nl$  of $B$, Thus we may put \newline

(3.4) $ v(t+ ln)= t, s+ kn = ws$ \newline

for some  integers  $v$ and $w$. So we may have  that  both  $t+ln $  and  $s+ kn$  are negative  integers. We  put  $c= v w$ for brevity.\newline

 On the other hand we have  the following equations successively; \newline

$1+ A^n= (1+ $$A \over c$$)^n $ $\Rightarrow c^n(1+ A^n)= (c+ A)^n \Rightarrow c^n\{1+ $$(v(t+ ln$)$ / s$$)^n\}= \{c+ (v(t+ln)/s)\}^n $
$\Rightarrow c^n\{s^n+ (v(t+ln))^n\}= (cs + v(t+ ln))^n \Rightarrow$ $c|v(t+ ln)$ ,which means that $c$  must  divide $v(t+ ln)= t$  by  (3.4). \newline

It follows that $B= (t+ ln)/(s+ kn)= v(t+ln)/v(s+ kn)=  v(t+ ln)/vws= v(t+ln)/ cs= t/ cs = d/ s $,where  $cd = v(t+ln)$ for some  positive integer $d$.\newline

Hence  we get  $s+ kn= -s$ and $t+ ln =- d$, so that  $2s= -kn$ and $t+ d= -ln$ are obtained  respectively.\newline

 It follows that   
\newline

(3.5) $n|s$    \newline

holds  since  $n \geq 3$,

which is a contradiction to our assumption that $X= s$  should not  be a  multiple  of  $n$.\newline

We may still have another contradiction  granting that  (3.5)  is  right.

If so,  then we  thus  obtain  that  $s^n \neq 1, (s+ kn ) \neq 1$  along with  the fact  $s^n \neq (s+ kn)^n$ ( if $n$  is odd).\newline

But then  from (3.1) we  get at   the  equation  $(s+ kn)^n(s^n + t^n) = s^n \{(s+ kn)+ (t+ ln )\}^n$  by substitution  of  the fractions , so that  $(-s)^n+ (-t)^n= (-s+ t+ ln)^n$. Hence  in the field  $\mathbb{Z}_n$ we have  $-s- t \equiv -s+ t+ln$, which reults in  $2t \equiv 0$  modulo $(n)$, and hence   \newline

(3.6)  $n|t $ \newline

because  $n \geq 3$.\newline

 Note that  
we  should  have $k \neq 0, l\neq 0$, $w= - 1$ and $v\leq - 2$  as  integers. At the same time from the congruence relation (3.2) , we have $c\equiv 1$ modulo $(n)$  as a  byproduct.\newline

After all  we are led to  contradictions  (3.5) and (3.6)  because  we assumed  that  
the fraction  $t/s= A$ is irreducible.\newline

Now  we consider the general case $B= (bt+ ln)/(bs+ kn)$= ($t+ $$\ln \over b$)/($s+$$\ kn \over b $),whose denominator and numerator are negative numbers.\newline

If $s+$$\ kn \over b$= $-s$, then we are led to contradiction likewise as above. \newline

Otherwise we have an equation of the form $1+(t/s)^n=[1+ (d/s')]^n$, where  $s'$ is a  divisor of $s$ ; so $s= ks'$ for some $k$ with $d$ and $s'$ relatively prime.  \newline

We may figure out  a general proof  as follows.\newline

 Consider  the equality $1= $$[(s'+ d) / s']^n$- $(t/s)^n$=$[(s'+ d)/s'- (t/s)][(s'+d)/s')^{n-1}+ ((s'+ d)/s')^{n-2}(t/s) \cdots +(t/s)^{n-1}]$ and the factorization of the right hand side of  this equation gives $ 0 <(ks'+ kd- t)/s <1$ and $(s+ kd - t)|s$.\newline

 If we take the reverse of both sides of the last equation, then we have\newline

(3.7) 1= $[s/ (ks'+kd- t)][s^{n-1}/((ks'+ kd)^{n-1}+ (ks'+ kd)^{n-2}t+
 \cdots +t^{n-1})]$. If we assume $(ks'+ kd- t )\geq 2$, then we meet  a contradiction $1 \equiv 0$ modulo $(s+ kd- t)$ since $(s+ kd- t)|s$.\newline

Hence we obtain $s+ kd - t= 1$, and so $kd \neq 1$. We thus have \newline

(3.8)1= $s^n/[( s+kd)^{n-1}+ (s+ kd)^{n-2}t+ \cdots + t^{n-1}]$,\newline

which gives rise to \newline

(3.9)1=$ s^n/ [(1+ t)^n- t^n$ \newline

by dint of  the identity $[(1+t)^{n-1}+ (1+t)^{n-2}t+ \cdots +t^{n-1}][(1+t)-t]= (1+t)^n- t^n$.\newline

From (3.9) it follows that 
$s^n= (1+t- kd)^n=( !+t)^n+ n(1+t)^{n-1}(-kd)+ \cdots+ (-kd)^n= (1+t)^n- t^n$ holds.\newline

Hence we get \newline

(4.0) $\binom{n}{1}(1+t)^{n-1}(-kd)+ \binom{n}{2}(1+t)^{n-2}(-kd)^2+ \cdots+ (-kd)^n= -t^n$.
Dividing both sides of (4.0) by $-nkd$, we obtain \newline

(4.1) $(1+t)^{n-1}[ \binom{n}{2}/n](1+t)^{n-2}(-kd)+ \cdots+ [(-kd)^{n-1}/n]= (t/kd)(t^{n-1}/n)$.\newline

Beware of the fact $kd\geq2$. So if we take both sides modulo some prime divisor of $d$, we are led to a contradiction\newline

$1\equiv 0$ after all because of the assumption
$n \geq 3$. \newline

So we see that  the equality $s+ kd- t = 1$  is impossible,\newline

which implies that  there is no solution of (2.1) at all.\newline

Hence we have completely proved  in this way  the Fermat's  Last theorem.\newline

\end{proof}

\begin{cor}
We have no solution of the designated equation (2.1)  even in  the field  $\mathbb{Q}$ of rational numbers.
\end{cor}

\begin{proof}

Immediate consequence of  the above theorem.
\end{proof}

\begin {cor}
$(1- a^n)^{1 \over n}$ for  any rational number  $a$  with  $0 < a  <1 $  must be  an  irrational  number. Likewise $(1+  a^n)^ {1\over n}$  for any nonzero  rational  number $a$  must be  an  irrational  number.

\end{cor}

\begin{proof}

If we divide  both sides of  (2.1)  by  $Z^n$, then  we get  \newline

(4.2) $x^n+ y^n= 1$, where  $0 \leq x= X/Z \leq 1$ and $0 \leq y= Y/Z \leq 1$.\newline

Because there is  no solution  of  rational number for the equation  $x^n= 1- y^n$, the  first  assertion for  $(1- a^n)^{1\over n}$  is  evident  on  the  one hand.\newline

On the other hand  if  such a  curve  on  the  $xy$-plane  met  with the line  $y= ax$ for  any  $a \in \mathbb {Q}$, then  we would  have  an equation of the form  $ x^n(1+  a^n)= 1$ , or  equivalently $x^n= (1+ a^n)^{-1}$.\newline

 But then there is no  rational solution $x$  of  this  last equation  by virtue of  the theorem 3.1. Hence  the  latter  assertion is also  evident.

\end{proof}

\begin{cor}
For any nonzero  rational  numbers  $x,y$, we have $x^n+ y^n\neq x^n y^n$.

\end{cor}

\begin{proof}
If we suppose  $x^n+ y^n=x^n y^n$, then  we  have $(x^n+ y^n)/ x^n  y^n = 1$ for any nonzero rational numbers $x,y$. \newline

Thus we get  $(1/x)^n+ (1/y)^n= 1$ ,which is nothing but  of the form (4.2). It is a  contradiction.

\end{proof}

\
\section{concluding remark}

\
\

We believe that there may be other ways to solve the Fermat's Last theorem. But this paper's way  is considered to be the easiest  among  them. Furthermore  we might wonder if  Fermat himself  solved  his theorem  in  this way. \newline

Also we think  that  most  undergraduate students majoring in mathematics  might understand this paper without difficulty. \newline

Moreover  It is  a marvelous  fact that  (2.1) has no solution even in the field  $\mathbb {Q}$  of  rational numbers, whose fact is just  nothing but  a  corollary  mentioned above of  the  Fermat's last theorem.\newline

As is well known, the equation of (2.1) is one thing  of  many sorts of  the so called  'Diophantine  equations'. \newline

Let's  take a look at some Diophantine equation  for a  moment such as \newline

(4.3) $ X^n\pm Y^n\pm  Z^n= 0$ ,\newline

where $X\neq 0,Y\neq 0,Z\neq 0$.\newline

It is not  difficult  to see  that  we can  transform this equation into  an equivalent Diophantine equation with  (2.1). \newline

Even if we are not luminaries in this area , we  could  absorb  ourselves  in  Diophantine equations.\newline

Because the proofs hereby up to now  are elementary,we didn't need  to cite bibliographies much  except for [KY] in that  it  shows  sort of  localization of  a  ring in the background material chapter.\newline

Acknowledgement.

We'd like to give many sincere thanks to all persons whom we are indebted to or who led us to good prosperity.

\bibliographystyle{amsalpha}

\end{document}